\documentclass[a4paper,12pt]{article}
\usepackage{amsmath,amsthm,amssymb,amsfonts}
\pdfoutput=1
\usepackage{leftidx}
\usepackage{bbm,bm}
\usepackage{latexsym}
\usepackage{mathrsfs}
\usepackage{threeparttable}
\usepackage{tabularx}
\usepackage{booktabs}
\usepackage{multicol,graphics}
\usepackage{subfigure}
\usepackage[dvips]{graphicx}
\usepackage{rotating}
\usepackage{epstopdf}
\usepackage{color}
\usepackage[dvips]{graphicx}
\usepackage{multicol,graphics}
\usepackage{color}
\usepackage{subfigure}
\usepackage{indentfirst}
\usepackage{geometry}
\usepackage{caption}
\usepackage{lineno}
\usepackage{enumerate,mdwlist}
\usepackage[numbers,sort&compress]{natbib}  
\usepackage{float}
\pdfoutput=1
\theoremstyle{definition}
\newtheorem{The}{Theorem}[section]    
\newtheorem{Lem}[The]{Lemma}
\newtheorem{Cor}[The]{Corollary}
\newtheorem{Pro}[The]{Proposition}

\theoremstyle{remark}

\theoremstyle{definition}

\numberwithin{equation}{section}

\begin{document}
\title{Relations between global forcing number and maximum anti-forcing number of a graph\footnote{This work is supported by NSFC (Grant No. 11871256)}}
\author{Yaxian Zhang and Heping Zhang\thanks{Corresponding author.}}
\date{{\small School of Mathematics and Statistics, Lanzhou University,
 Lanzhou, Gansu 730000, P.R. China}\\
{\small E-mails:\ yxzhang2016@lzu.edu.cn, zhanghp@lzu.edu.cn}}

\maketitle

\begin{abstract}
 The global forcing number of a graph $G$ is the minimal cardinality of an  edge subset discriminating  all perfect matchings of $G$, denoted by $gf(G)$.
 For any perfect matching $M$ of $G$, the minimal cardinality of an edge subset $S\subseteq E(G)\setminus M$ such that $G-S$ has a unique perfect matching is called the anti-forcing number of $M$, denoted by $af(G,M)$. The maximum anti-forcing number of $G$ among all perfect matchings is denoted by $Af(G)$. It is known that the maximum anti-forcing number of a hexagonal system equals the famous Fries number.

We are interested in some comparisons between the global forcing number and the maximum anti-forcing number of a graph. For a bipartite graph $G$,  we show that $gf(G)\geq Af(G)$. Next we mainly extend such result to non-bipartite graphs. Let $\mathcal{G}$ be the set of all graphs with a perfect matching which contain no two disjoint odd cycles such that their deletion results in a subgraph with a perfect matching.  For any $G\in \mathcal{G}$, we also have $gf(G)\geq Af(G)$ by revealing further property of non-bipartite graphs with a unique perfect matching. As a consequence, this relation also holds for the graphs whose perfect matching polytopes consist of non-negative $1$-regular vectors. In particular, for a brick $G$, de Carvalho, Lucchesi and Murty \cite{de2004perfect} showed that $G\in \mathcal{G}$  if and only if $G$ is solid, and if and only if its perfect matching polytope consists of non-negative $1$-regular vectors.

Finally, we obtain tight upper and lower bounds on $gf(G)-Af(G)$. For a connected bipartite graph $G$ with $2n$ vertices, we have that $0\leq gf(G)-Af(G)\leq \frac{1}{2}(n-1)(n-2)$;  For non-bipartite case, $-\frac{1}{2}(n^2-n-2)\leq gf(G)-Af(G)\leq (n-1)(n-2)$.

    \medskip
\end{abstract}
\section{Introduction}
    In this paper, we only consider finite simple graphs with at least one perfect matching. Let $G$ be a graph with vertex set $V(G)$ and edge set $E(G)$. We denote the order of $G$ by $v(G)=|V(G)|$, and the size  by $e(G)=|E(G)|$. An edge subset $M$ of $G$ is a \emph{perfect matching} (1-factor or a  \emph{Kekul\'e structure} in chemical literature), if every vertex of $G$ is incident with exactly one edge in $M$.

    In 1987, Klein and Randi\'c \cite{klein1987innate} introduced the \emph{innate degree of freedom} of a Kekul\'e structure, which plays an important role in the resonance theory in chemistry.  Harary et al. \cite{harary1991graphical} called it as forcing number:
    If a subset $S$ of a perfect matching $M$ in a graph $G$ is contained in exactly one perfect matching of $G$, then $S$ is a \emph{forcing set} of $M$. The minimum cardinality over all forcing sets of $M$ is called the \emph{forcing number} of $M$, denoted by $f(G,M)$. The minimum (resp. maximum) forcing number of $G$ is denoted by $f(G)$ (resp. $F(G)$). For a hypercube $Q_n$, Diwan \cite{Diwan2019} showed that $f(Q_n)=2^{n-2}$, confirming a conjecture by Pachter and Kim \cite{Pachter1998}.
    For a hexagonal system, Xu et al. \cite{xu2013maximum} and Zhou and Zhang \cite{Zhou2015forcing} showed that the maximum forcing number  is equal to its Clar number (resonant number). For polyomino graphs and (4,6)-fullerenes (also called Birkoff-Von Neumann graphs) the same result also holds (cf. \cite{zhang2016polyomino,Zhou2016polyomino,shi2017maximum}).

    As early as 1997, Li \cite{Li1997Hexagonal} raised the concept of \emph{forcing single edge} (i.e. anti-forcing edge). In 2007, Vuki\v{c}evi\'c and Trinajsti\'c \cite{vukiveevic2007anti,vukiveevic2008Kekule} introduced the \emph{anti-forcing number} of a graph $G$.
     In general, recently Lei et al. \cite{Lei2016Anti} and Klein and Rosenfeld \cite{Klein2014} independently defined the \emph{anti-forcing number} of a single perfect matching $M$ of a graph $G$.  A subset of $E(G)\setminus M$ is an anti-forcing set of $M$ if its  removal results in a graph with a unique perfect matching.
     The minimum cardinality of an anti-forcing set of $M$ is called the \emph{anti-forcing number} of $M$ in $G$, denoted by $af(G,M)$. The minimum (resp. maximum) anti-forcing number of $G$ is denoted by $af(G)$ (resp. $Af(G)$). Lei et al. \cite{Lei2016Anti}  and Shi et al. \cite{shi2017maximum} respectively showed that the maximum anti-forcing number of  hexagonal systems and (4,6)-fullerenes are  equal to their Fries numbers. For upper bounds on  maximum anti-forcing number of graphs, see \cite{Kai2017Anti,shi2017tight}

    The concept of ``global  forcing'' was introduced by Vuki\v cevi\'c \cite{vukivcevic2004total} to distinguish all perfect matchings of a graph. An edge subset $S$ of $G$ is called a \emph{global forcing set} of $G$, if no two distinct perfect matchings of $G$ coincide on $S$. The minimum cardinality of a global forcing set of $G$ is the \emph{global forcing number} of $G$, denoted by $gf(G)$. Since then, some scholars studied the global forcing numbers of some chemical graphs, see \cite{D.T2007Global,cai2012global,liu2014global,zhang2014global,sedlar2012global,vukivcevic2004total}.

    In this paper, we readily find that $gf(G)\geq F(G)$ for any graph $G$. A relation $Af(G)\geq F(G)$
    has been already revealed  \cite{Lei2016Anti}.  Do\v sli\'c \cite{D.T2007Global} and Deng and Zhang \cite{deng2017extremal} respectively proved that the global forcing number and the maximum forcing number  of a graph have the cyclomatic number as a common upper bound.
    Motivated by  the above facts, it is natural to consider relations between the global forcing number and the maximum anti-forcing number of a graph.

   For (4,6)-fullerene graph $G$, a plane cubic graph whose faces are hexagons and squares,  Cai and Zhang \cite{cai2012global} gave a sharp lower bound on  global forcing number:  $gf(G)\geq \lceil \frac{2f}{3}\rceil$, and Shi and Zhang \cite{shi2017maximum} obtained a formula  for the maximum anti-forcing number: $Af(G)= \left \lfloor \frac{v}{3} \right \rfloor+2 $, where  $f$ and $v$ are  the numbers of faces and vertices of $G$ respectively. By Euler's formula we have $2f=v+4$, which implies that $\lceil \frac{2f}{3}\rceil= \left\lfloor \frac{v}{3} \right \rfloor+2 $. Accordingly we have that $gf(G)\ge Af(G)$ for any (4,6)-fullerene graph $G$.

   For a general bipartite graph $G$ with $2n$ vertices, luckily we can prove that $gf(G)\geq Af(G)$ by finding a special edge $e$ of $G$ with $gf(G)\geq gf(G-e)+1$ and $Af(G-e)+1\geq Af(G)$. Moreover, we get that $gf(G)-Af(G)\leq \frac{1}{2}(n-1)(n-2)$,  and equality holds if and only if $G$ is isomorphic to $K_{n,n}$.

   The next  question is whether we can extend the above relation $gf(G)\geq Af(G)$ to non-bipartite graphs. In this paper we give a negative answer  by  constructing a matching covered graph $G_k$ for any positive integer $k$ such that $gf(G_k)-Af(G_k)=-k$. However we can extend such a relation to 
   conformal graphs which has a close connection with classical problems as perfect matching polytope and tight cut decomposition in matching theory of graphs \cite{lovasz2009matching,edmonds1982brick,edmonds1965maximum,lovasz1987matching}.

    Let $\mathcal{G}$ denote the set of graphs $G$ with a perfect matching such that $G$ contains no two disjoint odd cycles $C$ and $C'$ such that $G-C-C'$ has a perfect matching. We reveal a novel substructure of a graph with  a unique perfect matching $M$ and the minimum degree larger than $1$, which consists of two disjoint odd cycles and an odd path connecting them such that $M$ contains a perfect matching of it. Based on this substructure, we show that $gf(G)\geq Af(G)$ for a graph $G\in \mathcal{G}$ by finding a $1$-degree vertex from a special subgraph of $G$ with a unique perfect matching. As a corollary, we have that for a graph $G$ whose perfect matching polytope  consists  of non-negative $1$-regular vectors,  $G\in \mathcal{G}$ and thus  $gf(G)\geq Af(G)$.

    In particular, for a brick $G$ (3-connected bicritical graph),  de Carvalho, Lucchesi and Murty \cite{de2004perfect} showed that $G\in \mathcal{G}$  if and only if the perfect matching polytope of $G$ consists  of non-negative $1$-regular vectors, and if and only if $G$ is solid (without non-trivial separating cuts). For a general graph, the above three conditions are not necessarily equivalent.
    For more details  on solid bricks, see \cite{de2012Pfaffian,Lucchesi2018unsolved,edmonds1982brick,de2002conjecture, Chen2019solid, de2006brick}. In 2013, Kawarabayashi \cite{Kawarabayashi2013disjoint} gave a simpler proof for the characterization of the graphs without two disjoint odd cycles.

    Finally, for a connected graph $G$, we get tight upper and lower bounds on $gf(G)-Af(G)$: $-\frac{1}{2}(n^2-n-2)\leq gf(G)-Af(G)\leq (n-1)(n-2)$, where the second equality holds if $G$ is isomorphic to complete graph $K_{2n}$.


\section{Definitions and some preliminary results}
    \label{preliminaries}
    For a vertex $v$ in $G$, we denote the set of all neighbors of $v$ by $N_G(v)$. And the cardinality of $N_G(v)$ is said to be the \emph{degree} of $v$, denoted by $d_G{(v)}$. Let $\Delta(G)$ and $\delta(G)$ denote the maximum and minimum degree of $G$ respectively.
    Lei et al. \cite{Lei2016Anti} showed the following relations between the forcing number and the anti-forcing number of a perfect matching of a graph.
    \begin{The}
        {\rm \cite{Lei2016Anti}}
        \label{Af(G)<()F(G)}
        Let $G$ be a graph with a perfect matching $M$. Then
            $f(G,M)\leq af(G,M)\leq (\Delta(G)-1)f(G,M)$, and thus $F(G)\leq Af(G)\leq (\Delta(G)-1)F(G)$.
    \end{The}
     The  cyclomatic number of a connected graph $G$ is $c(G)=e(G)-v(G)+1$. In 2017, Deng and Zhang \cite{deng2017extremal} proved that the maximum forcing number of $G$ is less than or equal to $c(G)$.

    \begin{The}
   {\rm\cite{deng2017extremal}}
   \label{Af(G)<c(G)}
   Let $G$ be a connected graph with a perfect matching. Then $Af(G)\leq c(G)$. If $G$ is non-bipartite,
   then $Af(G)<c(G)$.
   \end{The}

   For  two sets $A$ and $B$, the \emph{symmetric difference} of $A$ and $B$ is  defined by $A\oplus B=A\cup B-A\cap B$.  Let $G$ be a graph with at least one perfect matching $M$. We say a subgraph $H$ of $G$ is \emph{nice} if $G-V(H)$ has a perfect matching. For convenience, we denote $G-V(H)$ by $G-H$. An even  cycle $C$ of $G$ is called an \emph{M-alternating cycle}, if the edges of $C$ appear alternately in $M$ or not. An even cycle  $C$ is a nice cycle of $G$ if and only if  $G$ has two perfect matchings $M_1$ and $M_2$ such that $M_1\oplus M_2=E(C)$. 
    In 2007, Do\v sli\'c \cite{D.T2007Global} gave the following result.
   \begin{The}
   {\rm\cite{D.T2007Global}}
   \label{gf(G)<c(G)}
   Let $G$ be a connected graph with a perfect matching. Then $gf(G)\leq c(G)$, and equality holds if and only if all cycles of $G$ are nice.
   \end{The}

     We denote the set of all perfect matchings of a graph $G$ by $\mathcal{M}$, and $\Phi(G)=|\mathcal{M}|$. Do\v sli\'c gave the following lower bound on the global forcing number.
    \begin{The}
        {\rm\cite{D.T2007Global}}
        \label{gf>phi(G)}
        Let $G$ be a graph with a perfect matching. Then $gf(G)\geq \lceil \log_2\Phi (G)\rceil$.
    \end{The}

   The following results give equivalent conditions for  forcing sets,  anti-forcing sets and  global forcing sets of a graph.
     \begin{Lem}
        {\rm\cite{Riddle2002The}}
        \label{forcing set}
        Let $G$ be a graph with a perfect matching $M$. Then an edge subset $S\subseteq M$ is a forcing set of $M$ if and only if $S$ contains an edge from every $M$-alternating cycle.
    \end{Lem}
    \begin{Lem}
        {\rm\cite{Kai2017Anti}}
        \label{anti-forcing}
         Let $G$ be a graph with a perfect matching $M$. Then an edge subset $S\subseteq E(G)\setminus M$ is an anti-forcing set of $M$ if and only if $S$ contains at least one edge of every $M$-alternating of $G$.
    \end{Lem}
    \begin{Lem}
        {\rm\cite{cai2012global}}
         \label{global forcing}
         Let $G$ be a graph with a perfect matching. Then an edge subset $S$ is a global forcing set of $G$ if and only if $S$ intersects each nice cycle of $G$.
    \end{Lem}

    From these equivalent definitions, we can find some correlations.

    \begin{The}
        \label{gf>F}
        Let $G$ be a graph with a perfect matching. Then $gf(G)\geq F(G)$.
    \end{The}
    \begin{proof}
       Let $S$ be a minimum global forcing set of $G$ and let $M$ be  a perfect matching  of $G$ with $f(G,M)=F(G)$. Then $gf(G)=|S|$, and $S$ intersects every $M$-alternating cycle by Lemma \ref{global forcing}.  If $S\subseteq M$, then $S$ is a forcing set of $M$ from Lemma \ref{forcing set}. Thus  $gf(G)=|S|\geq f(G,M)=F(G)$.
        Otherwise,  for each edge $e$ in $S$ not in $M$ there exists an edge $e'$ in $M$  which is adjacent to $e$. Replacing all edges $e$ in $S\setminus M$ with the edges $e'$ in $M$ and the other edges of $S$ remaining unchanged we get a new edge subset $S'$. Then  $S'\subseteq M$ and $S'$ intersect every $M$-alternating cycle. So $S'$ is a forcing set of $M$ by Lemma \ref{forcing set}, and
       $gf(G)= \left|S\right|\geq \left|S'\right|\geq f(G,M)=F(G)$.
    \end{proof}

    \begin{Cor}
        Let $G$ be a graph with a perfect matching. Then
            $gf(G)-Af(G)\geq (2-\Delta(G))F(G)$.
    \end{Cor}
    \begin{proof}
        It is immediate from  Theorems \ref{Af(G)<()F(G)} and \ref{gf>F}.
    \end{proof}

     Let $H$ be a subgraph of $G$. For any $F\subseteq E(G)$, we denote $F\cap E(H)$ by $F|_H$.
         \begin{Lem}
        {\rm\cite{cai2012global}}
        \label{nice subgraph}
        An edge subset $S$ is a global forcing set of a graph $G$ if and only if for each nice subgraph $H$ of $G$, $S|_H$ is a global forcing set of $H$.
    \end{Lem}

   \begin{Lem}
        {\rm\cite{cai2012global}}
         \label{GF.T}
         Let $G$ be a connected graph with a perfect matching. Then $S\subseteq E(G)$ is a minimum global forcing set of $G$ if and only if $T=G-S$ is a maximum (spanning) connected subgraph of $G$ without any nice cycle of $G$.
    \end{Lem}

    The lemma gives a characterization for  the complement of a minimum global forcing set $S$ of $G$. The following result shows that the addition of some edges in $G$ to $G-S$ results in a graph with a unique perfect matching.


\begin{Lem}
        \label{G-S+F}
        Let $G$ be a graph with a perfect matching. Then for any minimum global forcing set $S$ of $G$, we can find an edge subset $F\subseteq S$ such that $G-(S\setminus F)$ has a unique perfect matching.
    \end{Lem}
    \begin{proof}
        Let $T=G-S$. Then $T$ is a connected spanning subgraph of $G$, containing no nice cycles of $G$ by Lemma \ref{GF.T}. If $T$ has a perfect matching, then $T$ has exactly one perfect matching and we are done. Otherwise, we can pick two different perfect matchings of $T$, their symmetry difference forms at least one cycle of $T$, a nice cycle in $G$, a contradiction.

        So suppose that $T$ has no perfect matchings. But we can choose a minimum subset $F\subseteq S$ such that  $T'=T+F$  has a perfect matching.
        Next, we will show that $T'$ has exactly one perfect matching. If not, we can get two different perfect matchings of $T'$. Their symmetry difference contains a nice cycle $C$ of $T'$, which is also a nice cycle of $G$. By Lemma \ref{global forcing}, there exists an edge $e$ in $S\cap E(C)$. So $e\notin E(T)$ and $e\in F$.
        Let $F'=F-\left\{e\right\}$. Then $F'$ is smaller than $F$ but $T+F'$ still has a perfect matching, a contradiction.
    \end{proof}

    If $G$ has a 1-degree vertex $u$,   then $u$ is called a \emph{pendant vertex} and $uv$ a \emph{pendant edge}, where $v$ is the neighbor of $u$. The deletion of $u$ and $v$ with their incident edges is called a \emph{leaf matching operation} (simply  \emph{$LM$ operation}). Next, we can see that  $LM$ operation has  no effect on the minimum global forcing sets of a graph.

    \begin{Lem}
        \label{leaf operation}
        Let $G$ be a graph with a perfect matching. If $G'$ is a graph obtained from $G$ by a series of $LM$ operations, then the minimum global forcing sets of $G$ are the same as the ones of $G'$.
    \end{Lem}

    \begin{proof}If  $uv$ is a pendant edge of $G$, let $G'=G-\left\{u,v\right\}$.
         Since $uv$ belongs to all perfect matchings of $G'$,  it follows that a cycle of $G$ is nice in  $G$ if and only if it is also nice in $G'$. By Lemma \ref{global forcing}, the minimum global forcing sets of $G$ corresponds to that of $G'$. When we make repeatedly $LM$ operations from $G$, the same result holds always.
    \end{proof}

\section{Bipartite graphs}

 We  consider some relations between the global forcing number and the maximum anti-forcing number of graphs by starting bipartite graphs in this section.  The following structure of a bipartite graph with a unique perfect matching will play an important role.
    \begin{The}
      {\rm\cite{lovasz2009matching}}
      \label{B.U.PM}
      If $G=(U,V)$ is a bipartite graph with a unique perfect matching $M$ and $2n$ vertices, then we can label all vertices in $U$ and $V$
      such that $E(G)\subseteq \left\{u_iv_j|1\leq i \leq j\leq n\right\}$. Hence $G$ has pendant vertices $v_1$ and $u_n$, and $e(G)\leq \frac{1}{2}n(n+1)$, equality holds if and only if $E(G)=\left\{u_iv_j|1\leq i \leq j\leq n\right\}$.
    \end{The}
    Combining Theorem \ref{B.U.PM} and Lemma \ref{G-S+F}, we can obtain the following critical result that connects a minimum global forcing set of a bipartite graph $G$ with anti-forcing sets of a perfect matching of $G$.
    \begin{Lem}
        \label{bipartite}
           Let $G$ be a bipartite graph with at least two perfect matchings. For any perfect matching $M$ of $G$, $G$ has  a minimum global forcing set $S_0$ such that  $S_0\setminus M\neq \varnothing$.
    \end{Lem}

    \begin{proof} Let $G'$ be a graph obtained from $G$ by a series of LM operations so that $\delta(G')\geq 2$. Then $M':=M\cap E(G')$ is a perfect matching of $G'$.
        By Lemma \ref{leaf operation}, it suffices to prove that the result holds for $G'$ and $M'$.  Let $S$ be a minimum global forcing set of $G'$. Then $T=G'-S$ is a maximum  subgraph of $G$ without any nice cycle of $G$ by Lemma \ref{GF.T}. We claim that $T$ contains a pendant vertex.
        From Lemma \ref{G-S+F}, we can find an edge subset $F\subseteq S$ such that $T+F$ has a unique perfect matching. By Theorem \ref{B.U.PM}, we know that $T+F$ has a pendant vertex $u$. Hence $u$  is also a pendant vertex of $T$.
        Let $e\in E(T)$ and $e'\in M'$ be the two edges incident with $u$ in $G'$ ($e$ and $e'$ may be the same). Then $T_0=T-e+e'$ also contains no nice cycles of $G'$. Since   $T_0$ has the same size as $T$, by Lemma \ref{GF.T} we have that $S_0=E(G)-E(T_0)$ is a new minimum global forcing set. Since $\delta(G')\geq 2$, $S_0$ has an edge incident with $u$ other than $e'$ in $G'$, so  $S_0\setminus M'\neq \varnothing$.
    \end{proof}

    From the lemma, we can obtain a main result in comparing the global forcing number and the maximum anti-forcing number of a bipartite graph.

    \begin{The}
        \label{Gf>Af}
         Let $G$ be a bipartite graph with a perfect matching. Then $gf(G)\geq Af(G)$.
    \end{The}

    \begin{proof}
        We apply induction on $e(G)$. If $e(G)=\frac{v(G)}{2}$, then $gf(G)=Af(G)=0$ and we are done. Next we consider bipartite  graph $G$ with larger size. Suppose that for any bipartite graph with sizes smaller than $G$, the result holds. If $G$ has a unique perfect matching, then it is trivial as the initial case. Otherwise, by Lemma \ref{bipartite} we have that
        for any perfect matching $M$ of $G$ with $af(G,M)=Af(G)$, $G$ has a minimum global forcing set $S$ such that $S\setminus M\neq \varnothing$.  Take  an edge $e$ in $S\setminus M$, and let $G'=G-e$.

        We claim that  $Af(G')\geq Af(G)-1$.
         Since $M$ is also a perfect matching of $G'$, $Af(G')\geq af(G',M)$. Let $S'$ be a minimum anti-forcing set of $M$ in $G'$.
        Since any $M$-alternating cycle in $G$  either is contained in $G'$ or contains $e$,   $S'\cup \left\{e\right\}$ is an anti-forcing set of $M$ in $G$ by Lemma \ref{anti-forcing}. So
        \begin{displaymath}
            af(G',M)=\left|S'\right|=\left|S'\cup \left\{e\right\}\right|-1\geq af(G,M)-1=Af(G)-1,
        \end{displaymath}
        which implies that $Af(G')\geq af(G',M)\geq Af(G)-1$ and the claim holds.

        From Lemma \ref{nice subgraph}, we know that $S\setminus\{ e\}$ is a global forcing set of $G'$, which implies that $|S\setminus \{ e\}|\geq gf(G')$. By the induction hypothesis, we get that $gf(G')\geq Af(G')$. Thus
        \begin{displaymath}
            gf(G)=|S|=|S\setminus \{ e\}|+1\geq gf(G')+1\geq Af(G')+1\geq Af(G).
        \end{displaymath}
        \end{proof}
        \begin{Cor}
            \label{Af=c}
            If $G$ is a connected graph with $Af(G)=c(G)$, then $gf(G)=c(G)$.
        \end{Cor}

        \begin{proof}
            By Theorem \ref{Af(G)<c(G)}, we know that $G$ is a bipartite graph. So by Theorem \ref{Gf>Af} we have that
            $gf(G)\geq Af(G)=c(G)$. By Theorem \ref{gf(G)<c(G)},
            $gf(G)\leq c(G)$. So $gf(G)=c(G)$.
        \end{proof}
       Next we will fix the order of a graph, but ignore its size to estimate the difference between the global forcing number and the maximum anti-forcing number of bipartite graphs.

    \begin{The}
        \label{B.gf-Af<}
        If $G$ is a connected bipartite graph with a perfect matching and $2n\geq 6$ vertices, then
        \begin{displaymath}
            0\leq gf(G)-Af(G)\leq \frac{1}{2}(n-1)(n-2),
        \end{displaymath}
        and the right equality holds if and only if $G$ is isomorphic to $K_{n,n}$.
    \end{The}

    \begin{proof}
        By Theorem \ref{Gf>Af}, we directly get the left inequality. So we now consider the right inequality. For any perfect matching $M$ of $G$, we can get a maximum subgraph $T_M$ of $G$ such that $M$ is the unique perfect matching of $T_M$. So
        $af(G,M)=e(G)-e(T_M)$. Let $T$ be a maximum subgraph of $G$ without nice cycles of $G$. Then $gf(G)=e(G)-e(T)$ and $e(T)\geq 2n-1$ by Lemma \ref{GF.T}. By Theorem \ref{B.U.PM}, we have that $e(T_M)\leq \frac{1}{2}n(n+1)$. So
        \begin{eqnarray}
            \label{unequality}
                gf(G)-Af(G) &\leq& gf(G)-af(G,M)
                            = e(T_M)-e(T)\nonumber\\
                            &\leq & \frac{1}{2}n(n+1)-(2n-1)
                            = \frac{1}{2}(n-1)(n-2)\label{2}.
        \end{eqnarray}
       \begin{figure}[H]
            \centering
            \subfigure[$T_M$ and a nice cycle of $G$ missing exactly $u_n$ and $v_1$ for $n=10$.]{\label{f1_a}
                \includegraphics[width=0.7\textwidth]{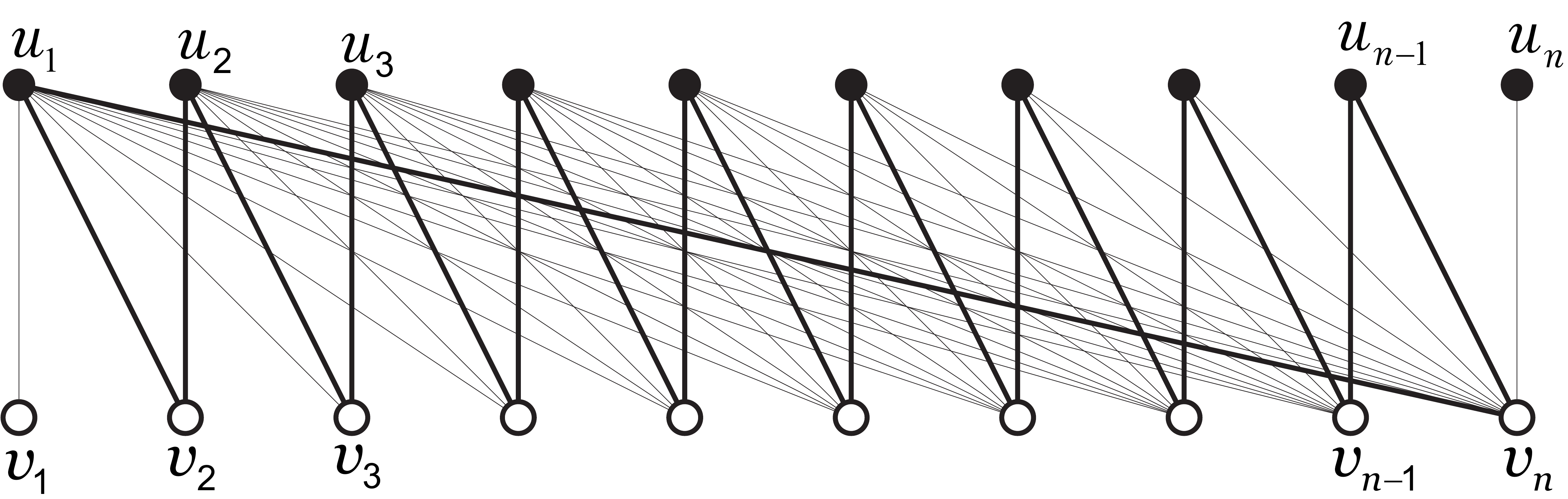}}\\
            \subfigure[A nice cycle of $G$ missing exactly $u_t$ and $v_r$ for $n=10$.]{\label{f1_d}
                \includegraphics[width=0.7\textwidth]{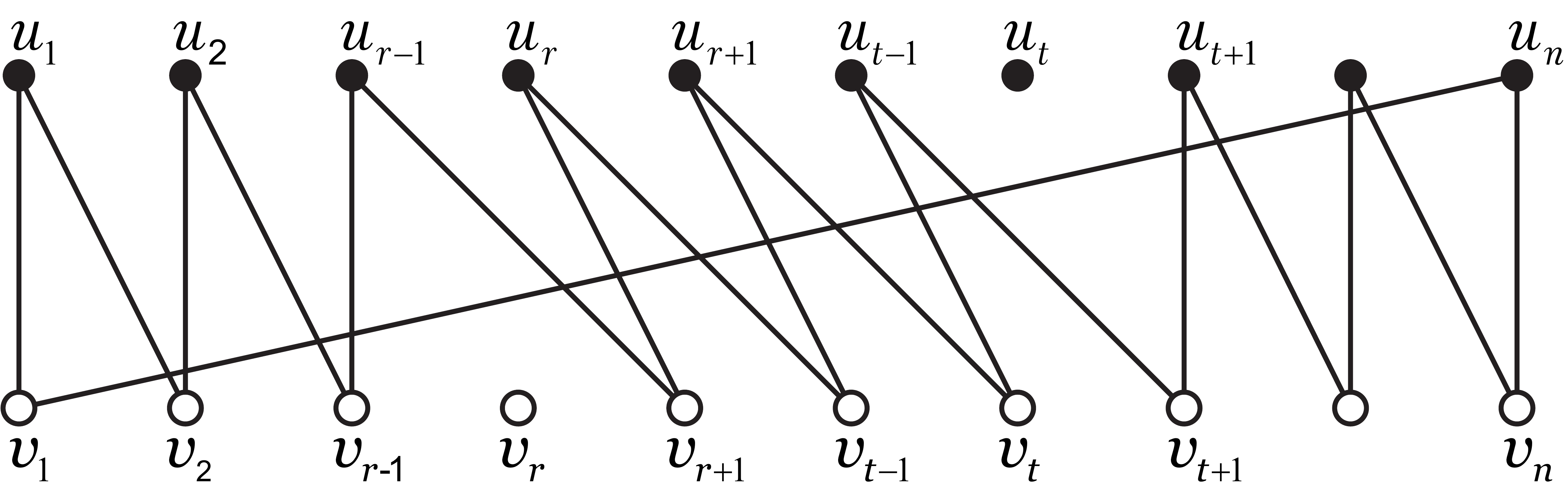}}\\
            \subfigure[A nice cycle of $G$ missing exactly $u_n$ and $v_r$ for $n=10$.]{\label{f1_b}
                \includegraphics[width=0.7\textwidth]{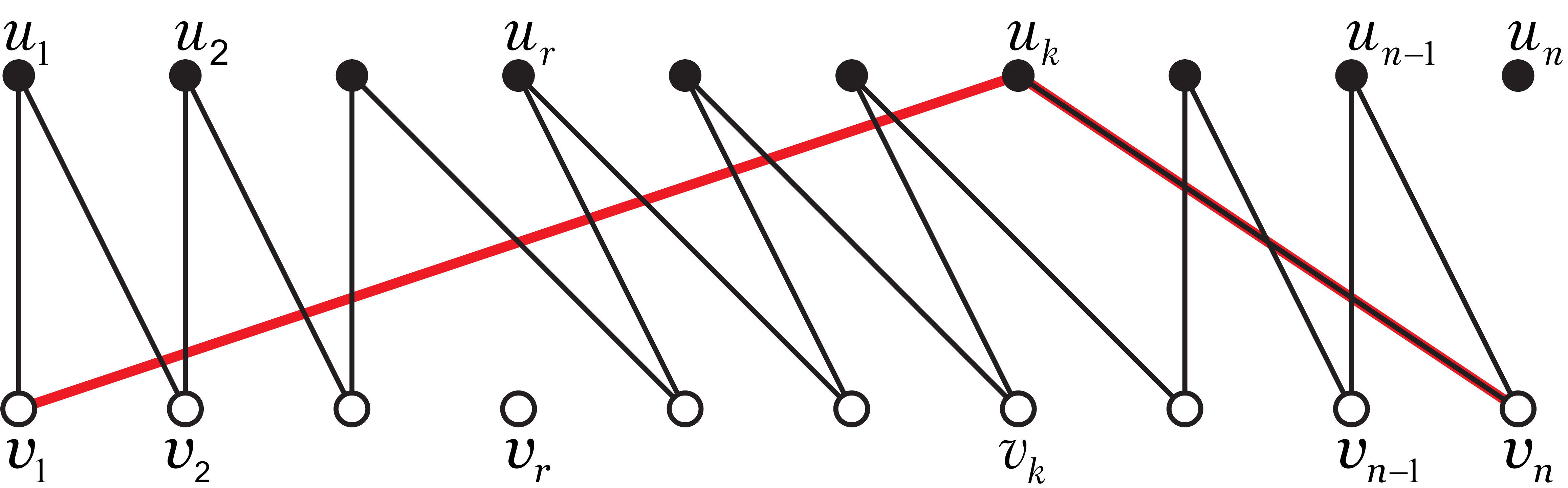}}
            \caption{Illustration for the proof of Theorem \ref{B.gf-Af<}}
        \end{figure}

        If $G$ is isomorphic to $K_{n,n}$, then by Theorem \ref{B.U.PM}, for every perfect matching $M$ of $G$, we have
        $e(T_M)=\frac{1}{2}n(n+1)$. Since every cycle of $G$ is nice, $T$ is a spanning tree of $G$, i.e. $e(T)=2n-1$. Thus $gf(G)-Af(G)=\frac{1}{2}(n-1)(n-2)$.

        Conversely, if $gf(G)-Af(G)=\frac{1}{2}(n-1)(n-2)$, then both equalities in Ineq. (\ref{2}) hold. so we have that $e(T)$ $=2n-1$ and for every perfect matching $M$ of $G$, $e(T_M)=\frac{1}{2}n(n+1)$. Next we want to show that $G$ is isomorphic to $K_{n,n}$, that is, $d_G(u_t)=n$ for each $1\leq t\leq n$. By Theorem \ref{gf(G)<c(G)}, every cycle of $G$ is nice. By Theorem \ref{B.U.PM}, we can label all vertices in $G$ such that $M=\left\{u_iv_i|1\leq i\leq n\right\}$ and $E(T_M)=\left\{u_iv_j|1\leq i\leq j\leq n\right\}$.
        Since $v_2u_2v_3u_3\ldots v_{n-1}u_{n-1}v_nu_1v_2$ is a nice cycle of $G$ missing exactly $u_n$ and $v_1$ (see Fig. \ref{f1_a}), $u_nv_1\in $ $E(G)$.

        For any $2\leq r< t\leq n-1$, $v_1u_1v_2u_2\ldots v_{r-1}u_{r-1}v_{r+1}u_rv_{r+2}u_{r+1}\ldots v_t u_{t-1}v_{t+1}u_{t+1}\ldots v_n$ $u_nv_1$ is a nice cycle of $G$ missing exactly $u_t$ and $v_r$ (see Fig. \ref{f1_d}), which implies that $u_tv_r\in E(G)$ and
        $G-\{u_1,v_1,u_n,v_n\}$ induces a complete bipartite graph. It remains to show that the degrees of $v_1$ and $u_n$ must be $n$.

        We can get a new perfect matching $M'$ by replacing the edges $u_1v_1$, $u_nv_n$ by $u_1v_n$ and $u_nv_1$ in $M$.
        Theorem \ref{B.U.PM} implies that $T_{M'}$ contains an edge not in $M'$ such that its two end vertices have degree $n$. So $G$ contains another $n$-degree vertex except for $u_1$ and $v_n$, say $u_k$.
        Then $u_n$ must have degree $n$, since
        $G$ contains a path from $v_1$ to $v_n$ and missing exactly $u_k$, $u_n$ and $v_r$, and a cycle missing exactly $u_n$ and $v_r$ formed by this path adding $v_1u_kv_n$ (see Fig. \ref{f1_b}) for any $2\leq r\leq n-1$.
    \end{proof}

     Shi et al. \cite{shi2017tight} gave an upper bound of the maximum anti-forcing number of graphs.

    \begin{The}
        {\rm\cite{shi2017tight}}
        \label{Af(G)<}
        Let $G$ be a graph with a perfect matching $M$. Then
        $af(G,M)\leq Af(G)\leq \frac{2e(G)-v(G)}{4}$.
    \end{The}

    A perfect matching $M$ of $G$ is said to be \emph{nice}, if $M$ satisfies all equalities in Theorem \ref{Af(G)<}.
    A characterization for a nice perfect matching of a graph was given as follows \cite{shi2017tight}.

    \begin{The}
        {\rm\cite{shi2017tight}}
        \label{N.P.M}
        For a perfect matching $M$ of a graph $G$, $M$ is nice if and only if for any two edges
        $e_1=xy$ and $e_2=uv$ in $M$, $xu\in E(G)$ if and only if $yv\in E(G)$, and
        $xv\in E(G)$ if and only if $yu\in E(G)$.
    \end{The}


    \begin{Cor}
        \label{gf(G)>n-1}
         Let $G$ be a connected graph with a nice perfect matching and $2n$ vertices. Then $gf(G)\geq n-1$.
    \end{Cor}

    \begin{proof}
        Let $H$ be a minimum connected spanning subgraph of $G$ with a nice perfect matching. Then by Theorem \ref{N.P.M}, $e(H)=n+2(n-1)=3n-2$, so $c(H)=n-1$. By Theorem \ref{Af(G)<}, $Af(H)=\frac{2e(H)-v(H)}{4}=n-1$. So $Af(H)=c(H)=n-1$.
        According to Corollary \ref{Af=c}, we know $gf(H)=n-1$. By Lemma \ref{nice subgraph}, we get
        $gf(G)\geq gf(H)=n-1$.
    \end{proof}

     In the following, we will give a graph $G$ such that $gf(G)-Af(G)$ is one less than the upper bound in Theorem \ref{B.gf-Af<}, i.e. $gf(G)-Af(G)=\frac{1}{2}(n^2-3n)$. A cycle $C$ in a graph $G$ is  a \emph{Hamilton cycle} if $V(C)=V(G)$.
    \begin{Pro}
        \label{K_n,n-e}
         Let $G$ be a connected bipartite graph with $2n\geq 8$ vertices. If $G$ is isomorphic to $K_{n,n}-e$ for any $e\in E(K_{n,n})$, then $gf(G)=n^2-2n-1$ and $gf(G)-Af(G)=\frac{1}{2}(n^2-3n)$. Conversely,  if $gf(G)=n^2-2n-1$, then $G$ is isomorphic to $K_{n,n}-e$.
    \end{Pro}

    \begin{proof}
          If $G$ is isomorphic to $K_{n,n}-e$ for any $e\in E(K_{n,n})$, let  $u$ and $v$ be nonadjacent vertices in different partite sets of $G$. Then $G-u-v$ is isomorphic to $K_{n-1,n-1}$. We claim that for any cycle $C$ in $G$, it is nice in $G$ if and only if it is not a Hamilton cycle of $G-u-v$. If $C$ is  a Hamilton cycle of $G-u-v$, it is not nice in $G$.  If either $u$ or $v$ belongs to $C$, then $G-C$ is either a null graph or a complete bipartite graph, so $C$ is a nice cycle of $G$. Otherwise, neither $u$ nor $v$ are in $C$, and  $G-C$ can be obtained from a complete bipartite with order at least $4$ by deleting an edge. Hence $G-C$ has a perfect matching, i.e. $C$ is a nice cycle of $G$, and the claim is verified.  Let $T$ be a maximum subgraph of $G$ without any nice cycle of $G$. Since  two Hamilton cycles of $G-u-v$ together contain a shorter cycle, which is a nice cycle of $G$, there is exactly one cycle in $T$, so $e(T)=2n$.
        By Lemma \ref{GF.T}, $gf(G)=e(G)-e(T)=(n^2-1)-2n=n^2-2n-1$.

        From Theorem \ref{B.gf-Af<}, we know that
        $Af(G)\geq gf(G)-\left[\frac{1}{2}(n^2-3n+2)-1\right]=\frac{1}{2}(n-2)(n+1)$.
        By Theorem \ref{Af(G)<}, we get that
        $Af(G)\leq \lfloor\frac{2e(G)-v(G)}{4}\rfloor=\lfloor\frac{n^2-n-1}{2}\rfloor=\frac{1}{2}(n-2)(n+1)$.
        So $Af(G)=\frac{1}{2}(n-2)(n+1)$ and $gf(G)-Af(G)=\frac{1}{2}(n^2-3n)$.

         Conversely, if $gf(G)=n^2-2n-1$, then $e(G)\geq gf(G)+v(G)-1=(n^2-2n-1)+2n-1=n^2-2$ by Theorem \ref{gf(G)<c(G)}. Because $G$ is bipartite, $e(G)\leq n^2$. If $e(G)=n^2$, then $G$ is isomorphic to $K_{n,n}$. By Theorem \ref{B.gf-Af<}, $gf(K_{n,n})=n^2-2n+1$, a contradiction.
        If $e(G)=n^2-2$, then $G\cong K_{n,n}-\left\{e,e'\right\}$ for distinct $e,e'\in E(K_{n,n})$. Since $gf(G)=c(G)$, by Theorem \ref{gf(G)<c(G)}, all cycles in $G$ are nice. If $e$ and $e'$ are adjacent, then $G-V(e)$ has a Hamilton cycle, not a nice cycle of $G$, a contradiction. So $e$ and $e'$ are not adjacent.

        Let $A=\left\{u_1,u_2,\ldots,u_n\right\}$ and $B=\left\{v_1,v_2,\ldots,v_n\right\}$ be a bipartite sets of $V(G)$.
        We can assume $e=u_1v_1$ and $e'=u_2v_2$. Let $P_1=u_4v_2u_3v_3u_2v_4$ and $P_2=u_4v_5u_5v_6\ldots u_{n-1}v_nu_nv_4$
        (see Fig. \ref{f2}). Obviously, $P_1\cup P_2$ is a cycle of $G$ missing exactly  $u_1$ and $v_1$, which is not a nice cycle of $G$, a contradiction. Thus, $e(G)=n^2-1$.
    \end{proof}
\begin{figure}[H]
            \centering
            \includegraphics[width=0.6\textwidth]{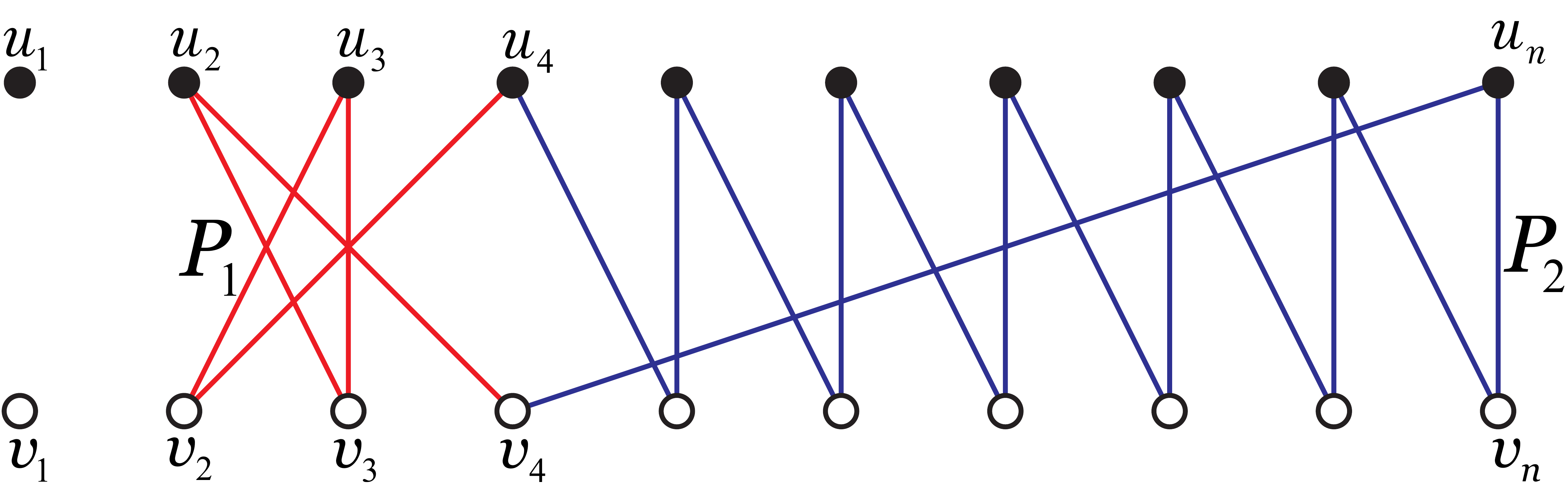}
            \caption{\label{f2}A cycle in $G$ missing $e$ and $e'$ for $n=10$}
        \end{figure}
   \section{Generalizations of Theorem \ref{Gf>Af}}
      A connected graph $G$ is called \emph{matching covered}, if every edge of $G$ belongs to a perfect matching of $G$.
      For a graph $G$, $\mathbb{R}^{E(G)}$ denotes the set of all vectors whose entries are indexed by the edges of $G$.
      For any perfect matching $M$ of $G$, the incidence vector of $M$ is denoted by $q^M$. The perfect matching polytope of $G$ is the convex hull of $\left\{q^M: M\in \mathcal{M}\right\}$, denoted by $PM(G)$. In a landmark paper, Edmonds \cite{edmonds1965maximum} gave a linear inequality description of $PM(G)$. For any $\boldsymbol{x}\in \mathbb{R}^{E(G)}$ and $F\subseteq E(G)$, let $\boldsymbol{x}(F)=\sum\limits_{e\in F}\boldsymbol{x}(e)$. For a nonempty subset $U\subset V(G)$, $\partial_G(U)$ represents for a {\em cut}, the set of  edges which have exactly one end vertex in $U$. If $U$ is a singleton, then $\partial_G(U)$ is a {\em trivial} cut.
      \begin{The}
         {\rm\cite{edmonds1965maximum}}
        \label{polytope definition}
        A vector $\boldsymbol{x}$ in $\mathbb{R}^{E(G)}$ belongs to $PM(G)$ if and only if it satisfies the following system of linear inequalities:
        \begin{eqnarray}
            \label{polytope character}
                 \boldsymbol{x} &\geq & 0\ \ \ \ \ \ \ \ \ \ \ \ \ \ \ \ \ \ \ \ \ \ \ \ \ \ \ \ \ \ \ \ \ \ (non-negativeity), \nonumber\\
                \boldsymbol{x}(\partial_G(v)) &=& 1\ \ \ for\ all\ v\in V(G) \ \ \ \ \ \ \ \ (degree\ constraints),\nonumber \\
                \boldsymbol{x}(\partial_G(S))&\geq & 1\ \ \ for\ all\ odd\ S\subset V(G) \ \ (odd\ set\ constraints).\nonumber
        \end{eqnarray}
        \end{The}
        A vector $\boldsymbol{x}$ in $\mathbb{R}^{E(G)}$ is \emph{1-regular} if $\boldsymbol{x}(\partial_G(v))=1$, for all $v\in V(G)$.
        For a cut $C:=\partial_G{(X)}$ of $G$, we denote the graph obtained from $G$ by shrinking the shore $\overline{X}=V(G)-X$ to a vertex by $G\left\{X\right\}$. We refer to $G\left\{X\right\}$ and $G\left\{\overline{X}\right\}$ as the \emph{$C$-contractions} of $G$. A cut $C$ of a matching covered graph $G$ is a \emph{separating cut} if both of the $C$-contractions of $G$ are also matching covered, and is \emph{tight} if $|C\cap M|=1$ for all
        $M\in \mathcal{M}$. Every tight cut is separating, but the converse is not true. A matching covered graph is \emph{solid} if all separating cuts of it are tight. A \emph{brick} is a non-bipartite matching covered graph satisfying that all tight cuts of it are trivial.

        \begin{The}
           {\rm\cite{de2004perfect}}
           \label{solid}
             For a brick $G$, the following three statements are equivalent.\\
           (i) $G$ is solid,\\
            (ii) $G\in \mathcal{G}$, and\\
            (iii) $PM(G)$ consists  of non-negative $1$-regular vectors.
        \end{The}
       de Carvalho et al. \cite{de2004perfect} pointed out that in 2003 Reed and Wakabayashi had already showed that (i) and (ii) are equivalent in Theorem \ref{solid}, and gave a new proof to it and proved the equivalence of (i) and (iii).

       In 1959, Kotzig \cite{Kotzig1959linear} showed the following property of the graphs with a unique perfect matching.

      \begin{The}
       {\rm\cite{Kotzig1959linear}}
        \label{cutline}
        If $G$ is a connected graph with a unique perfect matching, then $G$ has a cut-edge belonging to the perfect matching.
      \end{The}
       For the graphs $G$  with a unique perfect matching,  we come to a deeper structure property: $G$ contains either a pendant vertex as Theorem \ref{B.U.PM} in the bipartite case or an {\em odd dumbbell} (see Fig. \ref{f3}): two disjoint odd cycles connected by one path of odd length.
       \begin{figure}[H]
            \centering
            \includegraphics[width=0.5\textwidth]{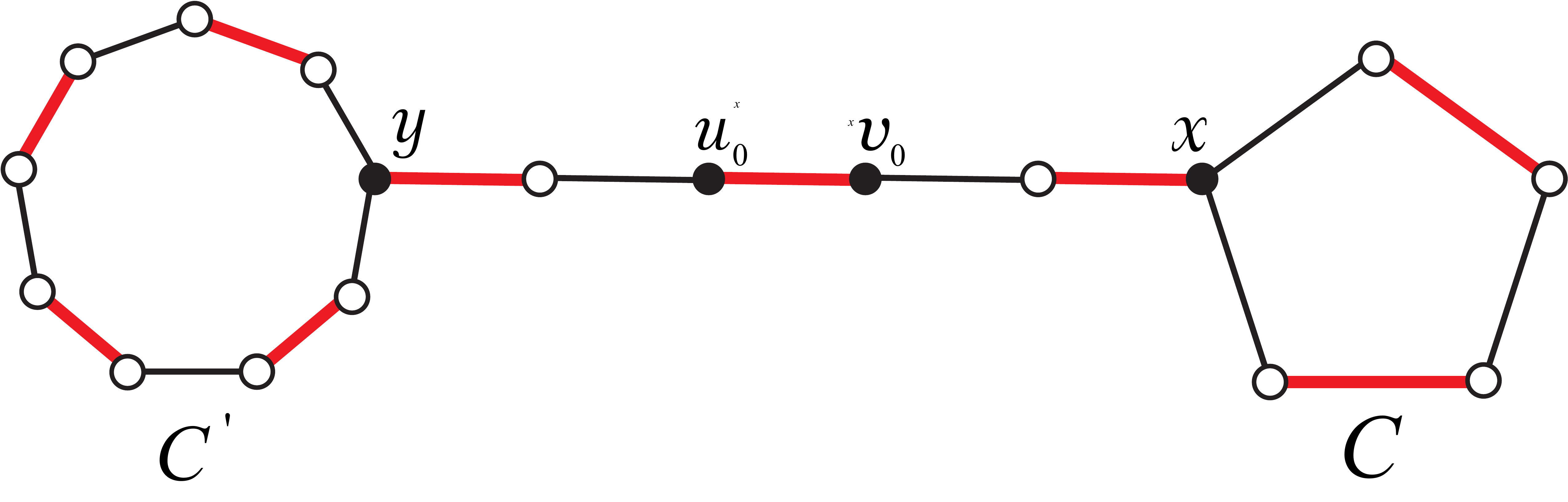}
            \caption{\label{f3} An odd dumbbell: The bold (red) edges form the unique perfect matching.}
        \end{figure}
       \begin{The}
         \label{disjoint}
         Let a graph $G$ have a unique perfect matching $M$ without pendant vertices.  Then $G$ contains an odd dumbbell as a  subgraph so that  $M$ contains the perfect matching of the subgraph.
      \end{The}
      \begin{proof}
           By Theorem \ref{cutline}, $G$ has a cut edge $uv\in M$. Since $\delta(G)\geq 2$, the components $G_u$ and  $G_v$ of $G-uv$ containing $u$ and $v$ respectively, must be non-trivial. Consider the longest $M$-alternating path in each component, starting with $u$ and $v$, respectively. This must end with an edge in $M$, and the last vertex must be adjacent to some vertex in the path. This must form an odd cycle and gives the required subgraph.
      \end{proof}

       Now we can extend the relation $gf(G)\geq Af(G)$ to the class of  graphs $\cal G$. First we extend Lemma \ref{bipartite} as follows.

       \begin{Lem}Let $G\in \mathcal{G}$ have at least two perfect matchings. Then for any perfect matching $M$ of $G$, there exists a minimum global forcing set $S_0$ such that  $S_0\setminus M\neq\varnothing$.
       \end{Lem}
       \begin{proof} We use almost the same proof as Lemma \ref{bipartite}, but only replace Theorem \ref{B.U.PM} with Theorem \ref{disjoint}. For the present graph $G$, $T+F$ has a unique perfect matching. Although $T+F$ is not necessarily bipartite,  $T+F$ is in $\cal G$. By Theorem \ref{disjoint}, $T+F$ has a pendant vertex. The others remain unchanged.
       \end{proof}
       Then by using the above lemma we obtain the following generalization as the proof of Theorem \ref{Gf>Af} .

      \begin{The}
         \label{Ggf>Af}
         If $G\in \mathcal{G}$, then $gf(G)\geq Af(G)$.
      \end{The}
         Combining Theorems \ref{solid} and \ref{Ggf>Af}, we  directly get the following result.
      \begin{Cor}
         If $G$ is a solid brick, then $gf(G)\geq Af(G)$.
      \end{Cor}
      For general non-bipartite graphs, the three conditions in Theorem \ref{solid} are not equivalent.
      However  we have the following implication.

        \begin{Lem}
            \label{polytope}
             If  $PM(G)$ consists  of non-negative $1$-regular vectors, then $G\in \mathcal{G}$.
        \end{Lem}

        \begin{proof}
           Suppose to the contrary that $G$ has disjoint odd cycles $C$ and $C'$ such that $G-C-C'$ has a perfect matching $M$.
           Let $\boldsymbol{x}$ be a vector in $\mathbb{R}^{E(G)}$ such that $\boldsymbol{x}(e)=1$, if $e\in M$; $\boldsymbol{x}(e)=\frac{1}{2}$, if $e\in C\cup C'$; $\boldsymbol{x}(e)=0$, otherwise. Note that $\boldsymbol{x}$ is a non-negative and $1$-regular vector. So $\boldsymbol{x}\in PM(G)$.  But for odd cycle $C$, $\boldsymbol{x}(\partial_G(C))=0$,  contradicting  the odd set constraint in  Theorem \ref{polytope definition}.
        \end{proof}
         By Theorem \ref{Ggf>Af} and Lemma \ref{polytope}, we can get the following corollary.
        \begin{Cor}
             If $PM(G)$ consists  of non-negative $1$-regular vectors, then $gf(G)\geq Af(G)$.
        \end{Cor}
      An example in Fig. \ref{f4} shows  there exists a matching covered graph $G$  not in $ \mathcal{G}$ with $gf(G)\geq Af(G)$.
        \begin{figure}[H]
            \centering
            \includegraphics[width=0.2\textwidth]{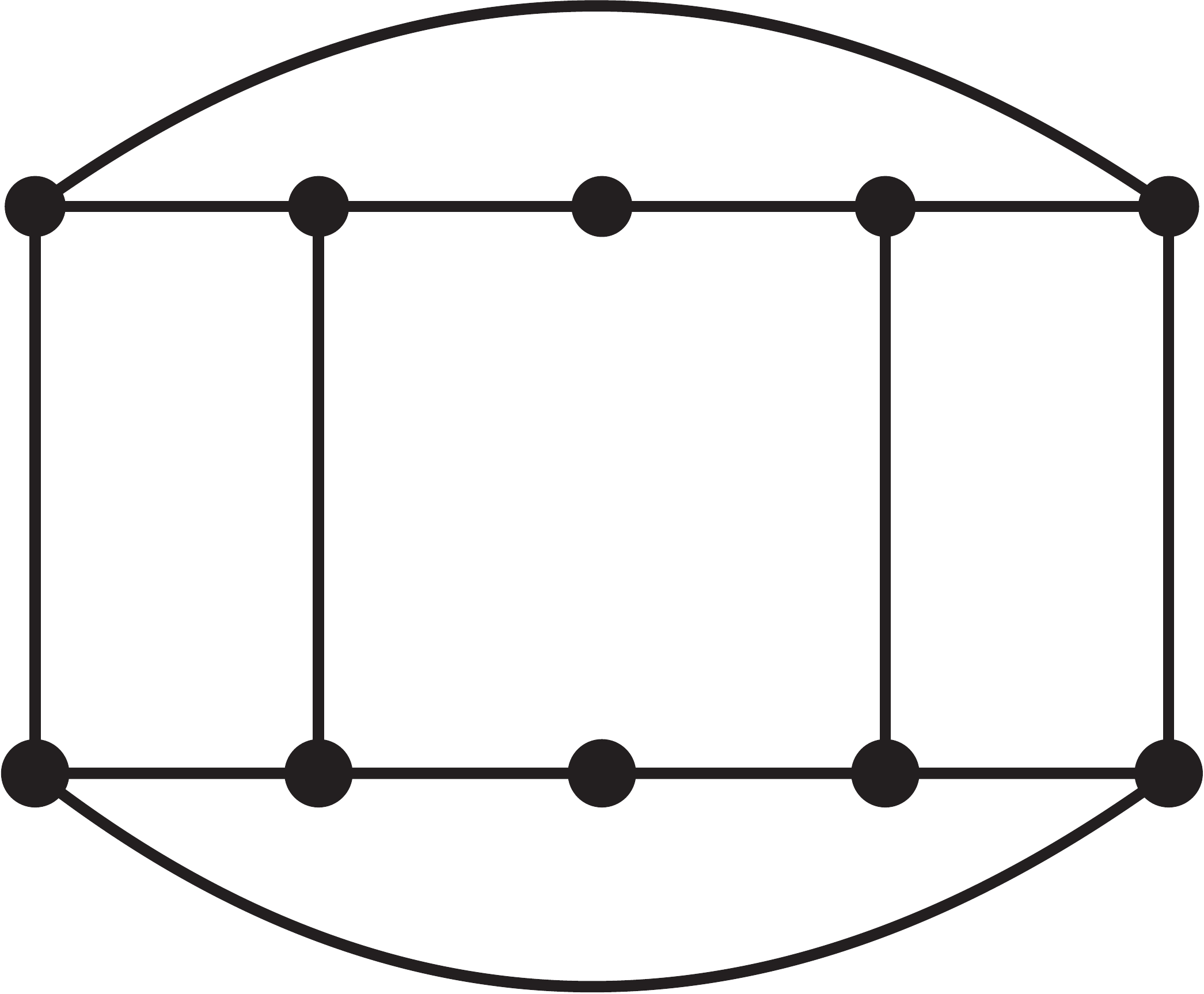}
            \caption{\label{f4}  A graph $G\notin \mathcal{G}$, and $gf(G)=Af(G)=3$.}
        \end{figure}
\section{The difference $gf(G)-Af(G)$}

    In this section, we will give sharp lower and upper bounds on  the difference $gf(G)-Af(G)$ on general connected graphs $G$.
    \begin{The}
        {\rm\cite{lovasz2009matching}}
        \label{G.M.E}
        If $G$ has a unique perfect matching and $2n$ vertices, then $e(G)\leq n^2$.
    \end{The}

    \begin{The}
        \label{gf-Af(b)}
        Let $G$ be a  connected graph with a perfect matching and $2n\geq 4$ vertices. Then
        \begin{displaymath}
            -\frac{1}{2}(n^2-n-2)\leq gf(G)-Af(G)\leq (n-1)(n-2),
        \end{displaymath}
        and the left equality holds if and only if $n=2$.
    \end{The}

    \begin{proof}
    ($1$) We first prove the left inequality. If $gf(G)=0$, then $G$ has a unique perfect matching, so $Af(G)=0$. If $gf(G)=1$, then $G$ has at least two perfect matchings. By Theorem \ref{gf>phi(G)}, we have that $\Phi(G)\leq 2^{gf(G)}=2$. Hence, $G$ has exactly two perfect matchings, denoted by $M_1$ and $M_2$. Obviously, $Af(G)=af(G,M_1)=af(G,M_2)=1$. So for the case of $gf(G)\leq 1$, we have that $-\frac{1}{2}(n^2-n-2)\leq 0=gf(G)-Af(G)$.

    Next we consider $gf(G)\geq 2$. Let $T$ be a maximum  subgraph of $G$ without any nice cycle of $G$. By Lemma \ref{G-S+F}, we can find an edge subset $F\subseteq E(G)\setminus E(T)$ such that $T+F$ has a unique perfect matching.
    From Theorem \ref{G.M.E}, $e(T)\leq e(T+F)\leq n^2$. By Lemma \ref{GF.T} and Theorem \ref{Af(G)<}, we have
    \begin{eqnarray}
        gf(G)-Af(G) &\geq & gf(G)-\frac{2e(G)-v(G)}{4}
        =gf(G)-\frac{gf(G)+e(T)}{2}+\frac{n}{2}\nonumber \\
                    &=& \frac{1}{2}gf(G)-\frac{1}{2}e(T)+\frac{n}{2}
                    \geq  -\frac{1}{2}(n^2-n-2)\label{4}.
    \end{eqnarray}

    Now we verify the equivalence condition for the left equality. If $n=2$, for all cases of $G$, we can verify that $gf(G)-Af(G)=0$. Conversely, if $gf(G)-Af(G)=-\frac{1}{2}(n^2-n-2)$, then for $gf(G)\leq 1$, we know $gf(G)-Af(G)=0$, which implies that $n=2$. For
    $gf(G)\geq 2$, all the equalities in Ineq. (\ref{4}) hold, so $Af(G)=\frac{2e(G)-v(G)}{4}$ and $gf(G)=2$.
    By Corollary \ref{gf(G)>n-1}, $n\leq gf(G)+1=3$. Next we will prove that if $n=3$, then $gf(G)\neq 2$. For $n=3$, we have $gf(G)-Af(G)=-2$, so $Af(G)=4$, $e(G)=11$ and $c(G)=6$. By Theorem \ref{N.P.M}, $G$ has two cases $G_1$ and $G_2$ as shown in Fig. $5$ in the sense of isomorphism, where $\{u_1u_2,u_3u_4,u_5u_6\}$ is a nice perfect matching of $G$.
    \begin{figure}[H]
        \centering
        \subfigure[Graph $G_1$]{\label{f5_a}
            \includegraphics[width=0.2\textwidth]{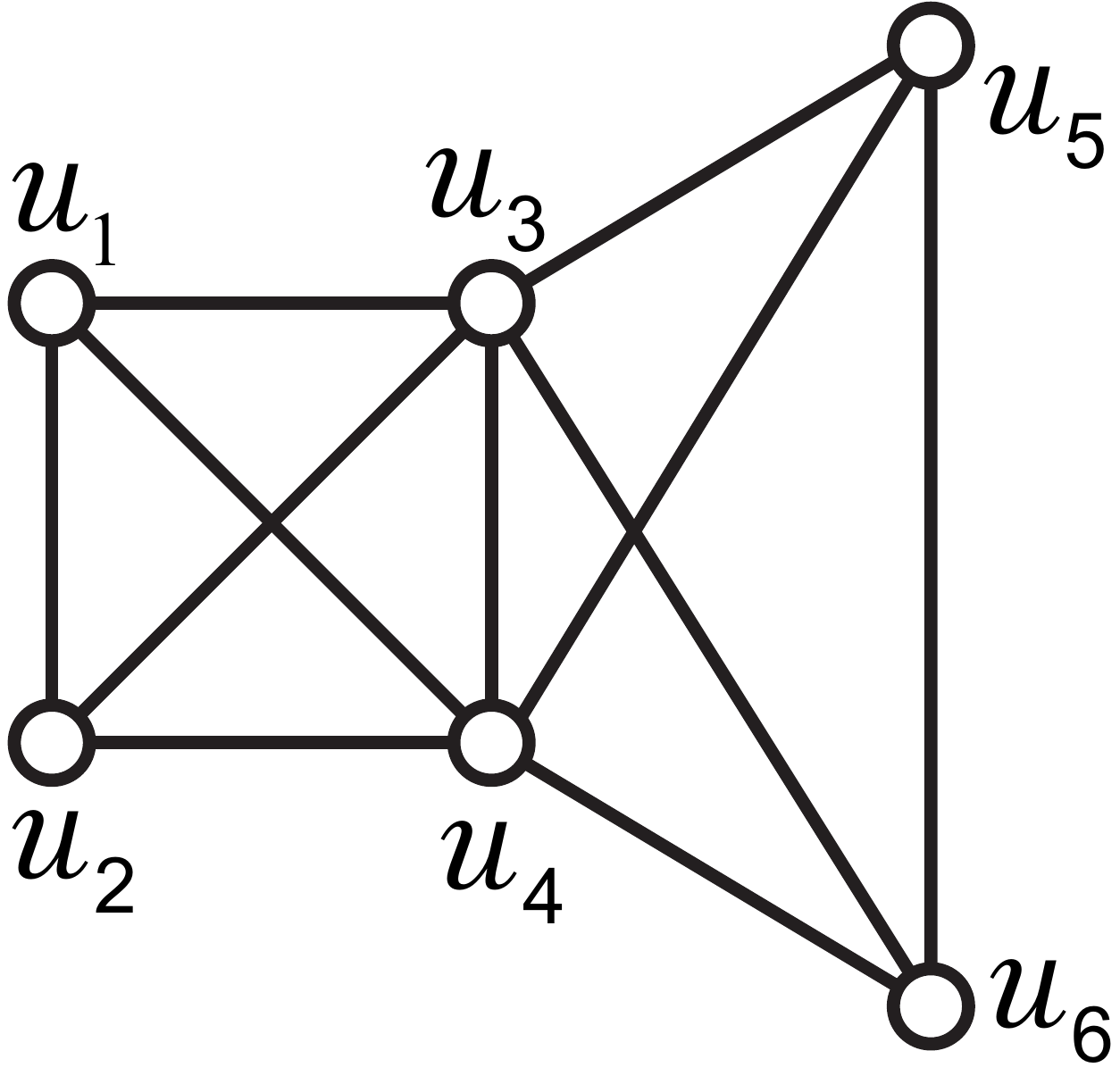}}\ \ \ \ \ \ \ \ \ \ \ \
        \subfigure[Graph $G_2$]{\label{f5_b}
            \includegraphics[width=0.2\textwidth]{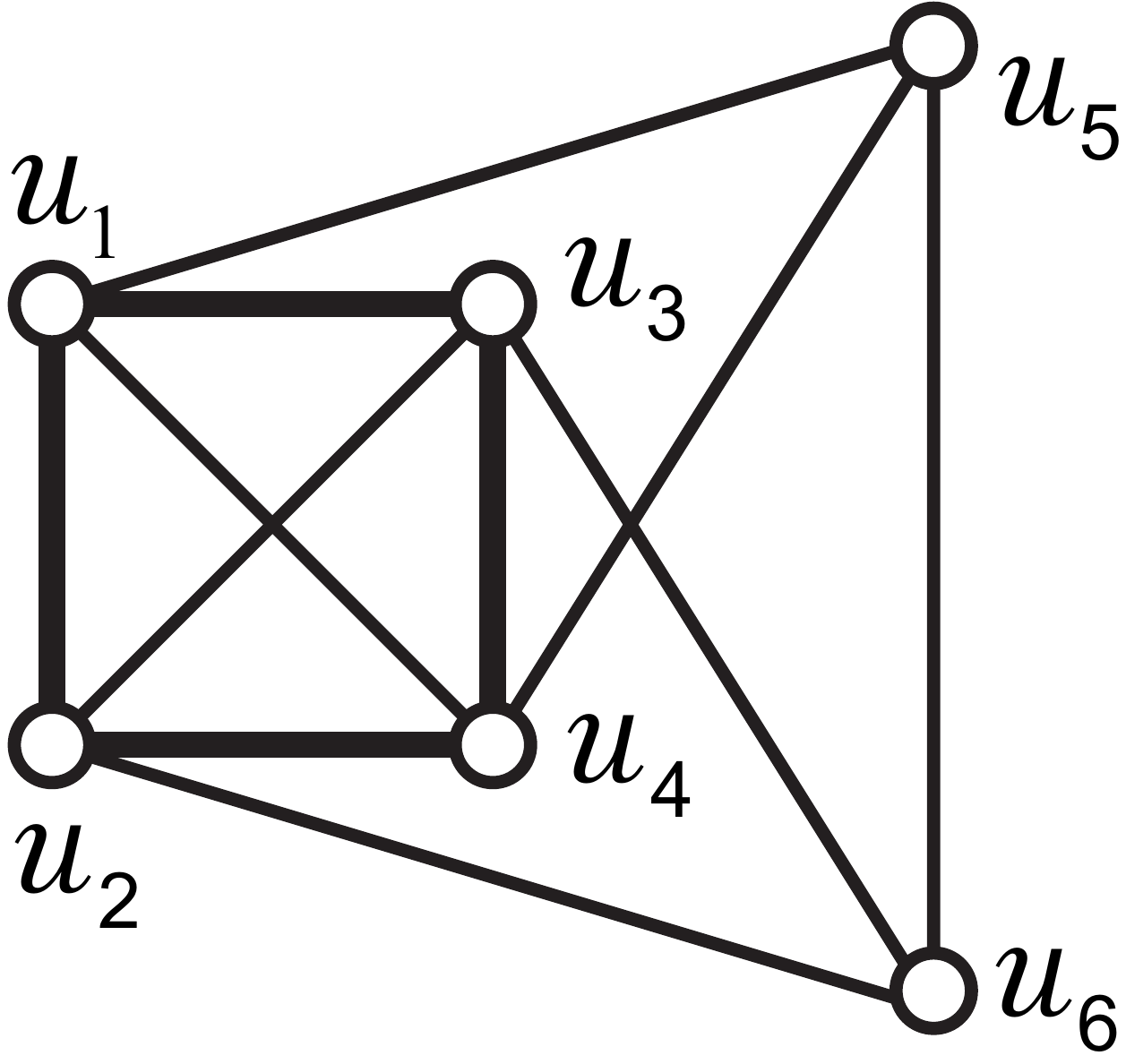}}
        \caption{Two cases of $G$ with a nice perfect matching, $n=3$ and $e(G)=11$.}
    \end{figure}
     We first consider the graph $G_1$. Let $S_1$ be a minimum global forcing set of $G_1$. Then $|S_1\cap E(G')|\geq 2$, for the subgraph $G'$ of $G_1$ induced by $\left\{u_1,u_2,u_3,u_4\right\}$, which is isomorphic to $K_4$. Similarly, the subgraph $G''$ of $G_1$ induced by $\left\{u_3,u_4,u_5,u_6\right\}$ satisfies that $|S_1\cap E(G'')|\geq 2$. Thus we have $gf(G_1)=|S_1|\geq |S_1\cap E(G')|+|S_1\cap E(G'')|-1\geq 3$.

    Next we claim that $gf(G_2)=4$.
    Since $\{u_1u_2,u_2u_4,u_3u_4,u_1u_3\}$ is a global forcing set of $G_2$, $gf(G_2)\leq 4$. Suppose to the contrary that $gf(G_2)\leq 3$. Let $S_2$ be a minimum global forcing set of $G_2$. Then $T_1=G_2-S_2$ contains at least three cycles. Note that every even cycle in $G_2$ is nice, with length either $4$ or $6$. And any three triangles in $G_2$ must produce a $4$-cycle. So $T_1$ has a pentagon $C$. Let $C'$ be another cycle in $T_1$ except for $C$. Then $|V(C)\cap V(C')|\geq 2$. So we can take two consecutive common vertices of $C$ and $C'$ such that there exists three internally disjoint paths between them, which must contain an even cycle, a contradiction. Thus $gf(G_2)=4$.

    ($2$) We now prove the right inequality.
    For any perfect matching $M$ of $G$, let $T_M$ be a maximum subgraph of $G$ with a unique perfect matching $M$.  Then $af(G,M)=e(G)-e(T_M)$.
    Corollary \ref{G.M.E} implies that $2n-1\leq e(T_M)\leq n^2$.
    Let $T$ be a maximum spanning subgraph of $G$ without nice cycles of $G$. Then $e(T)\geq 2n-1$. If $e(T_M)\leq n^2-n+1$, then
        $gf(G)-Af(G)\leq gf(G)-af(G,M)=e(T_M)-e(T)\leq (n^2-n+1)-(2n-1)=n^2-3n+2$,
    and the required result holds.

    Next we consider the case $e(T_M)\geq n^2-n+2$. Let $e(T_M)=n^2-k$, where $0\leq k\leq n-2$.
    By the definition of $T_M$, the subgraph of $T_M$ induced by $\{u_i,v_i,u_j,v_j\}$ can contain at most 4 edges. Consider the subgraphs induced by $\{u_i,v_i,u_{i+1},v_{i+1}\}$ for $1\leq i\leq n-1$. If $k+1$ of these subgraphs contain only 3 edges, the number of edges in $T_M$ is at most $n^2-(k+1)$, a contradiction. Therefore $n-1-k$ of these must have 4 edges and give $n-1-k$ triangles whose edges can not form additional cycles. So we have $e(T)\geq 2n-1+n-1-k=3n-2-k$.
    It follows that $gf(G)-Af(G)\leq e(T_M)-e(T)\leq (n^2-k)-(3n-2-k)=n^2-3n+2$ and we are done.
\end{proof}

The following discussions will show that the upper bound on $gf(G)-Af(G)$ in Theorem \ref{gf-Af(b)} can be achieved.

\begin{Pro}
   \label{K_{2n}}
   $gf(K_{2n})=2(n-1)^2$ and $gf(K_{2n})-Af(K_{2n})=(n-1)(n-2)$.
\end{Pro}

\begin{proof}
   Let $T$ be a maximum spanning subgraph of $K_{2n}$ which contains no any nice cycle of $K_{2n}$. By Lemma \ref{GF.T}, $gf(K_{2n})=e(K_{2n})-e(T)$. Obviously, every even cycle in $K_{2n}$ is nice. So $T$ has no even cycles. We claim that $e(T)\leq 3n-2$.
   Let $C_1,C_2,\ldots,C_{c(T)}$ be different cycles in $T$. Then they are pairwise edge-disjoint. Otherwise, there are distinct $C_i$ and $C_j$ sharing an edge. So we can take a path $P$ on $C_i$, which has only two end vertices and no edges in $C_j$.
   Then $C_j\cup P$ consists of three internally disjoint paths between two vertices, which must contain an even cycle, a contradiction.
   For any $1\leq i\leq c(T)$, we have $e(C_i)\geq 3$. So
   \begin{displaymath}
   e(T)=c(T)+2n-1\geq \sum\limits_{i=1}^{c(T)}\left|E(C_i)\right|\geq 3c(T),
   \end{displaymath}
   which implies that $c(T)\leq n-1$ and $e(T)\leq 3n-2$. So the claim holds.
   \begin{figure}[H]
   \centering
       \includegraphics[width=0.6\textwidth]{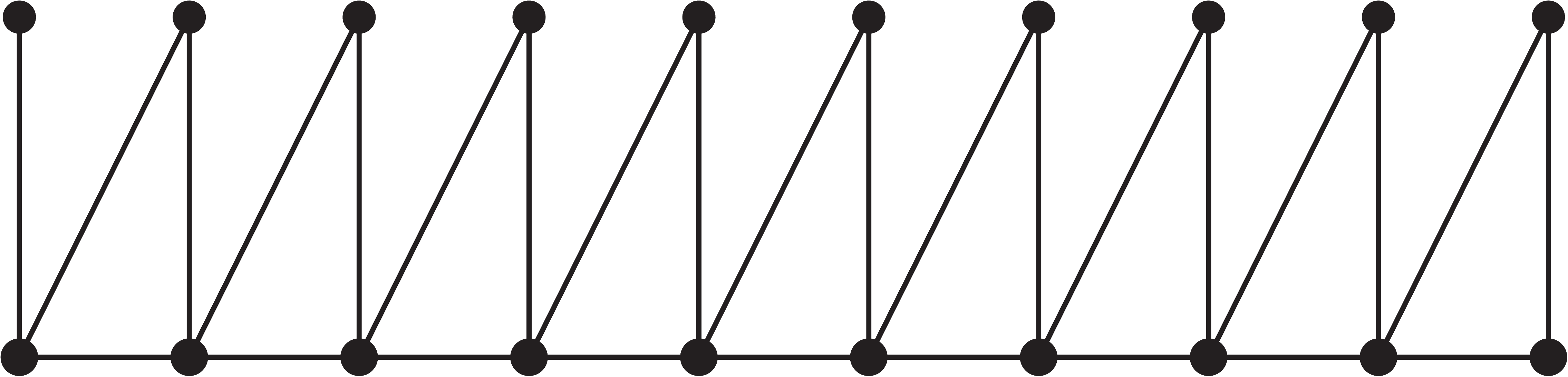}
   \caption{\label{f6}A spanning subgraph without any nice cycles of $K_{2n}$ for $n=10$.}
   \end{figure}
   On the other hand, we can find $n-1$ edge-disjoint triangles in $K_{2n}$ as shown in Fig. \ref{f6}, so $e(T)\geq 3n-2$.
   Moreover, $e(T)=3n-2$ and $gf(K_{2n})=2(n-1)^2$. It is known  \cite{shi2017tight}  that $Af(K_{2n})=n^2-n$, so $gf(K_{2n})-Af(K_{2n})=(n-1)(n-2)$.
   \end{proof}

    Finally, for any positive integer $k$ we will construct a matching covered graph $G_k$ such that $gf(G_k)-Af(G_k)=-k$.
    Let $T_i$ be a copy of the triangular prism, where $1\leq i\leq k$. Let $u_1^{(i)}u_2^{(i)}u_3^{(i)}u_1^{(i)}$ and  $u_4^{(i)}u_5^{(i)}u_6^{(i)}u_4^{(i)}$ be two triangles in $T_i$, and $u_1^{(i)}u_4^{(i)}$, $u_2^{(i)}u_5^{(i)}$ and $u_3^{(i)}u_6^{(i)}$ are three remaining edges in $T_i$. We join consecutively $T_1,T_2,\ldots,T_k$ with edges $u_1^{(i)}u_1^{(i+1)}$ and  $u_4^{(i)}u_4^{(i+1)}$, for each $1\leq i\leq k-1$, resulting in a graph $G_k$ (see Fig. \ref{f7_a}). Dr. Kai Deng gave $gf(G_1)-Af(G_1)=-1$. In general we have the following result.
      \begin{figure}[!htbp]
    \centering
    \subfigure[Graph $G_k$ for $k=3$]{\label{f7_a}
        \includegraphics[width=0.3\textwidth]{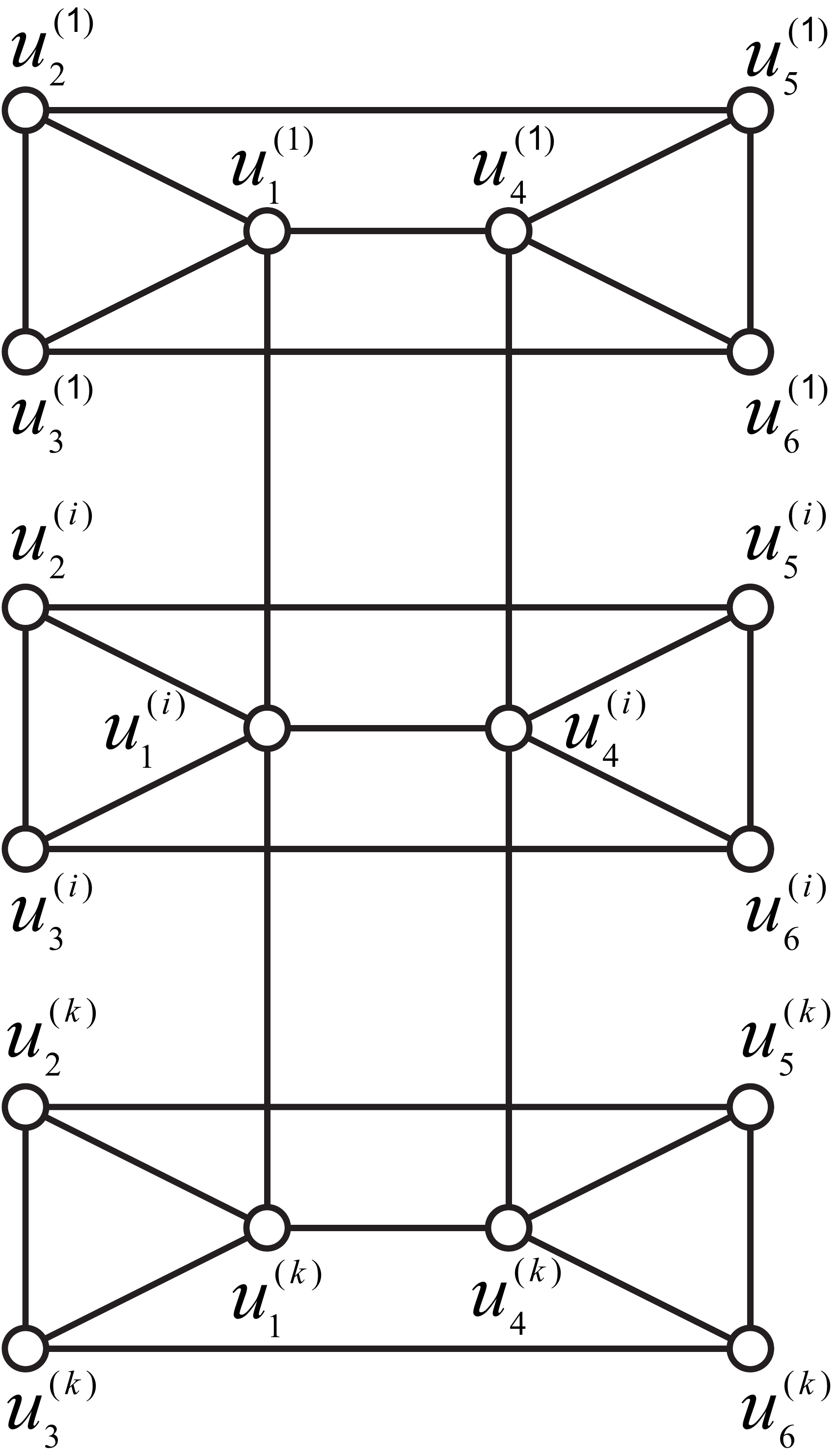}}\ \ \ \ \ \ \ \ \ \ \
    \subfigure[Graph $G_k-F$ for $k=3$]{\label{f7_b}
        \includegraphics[width=0.3\textwidth]{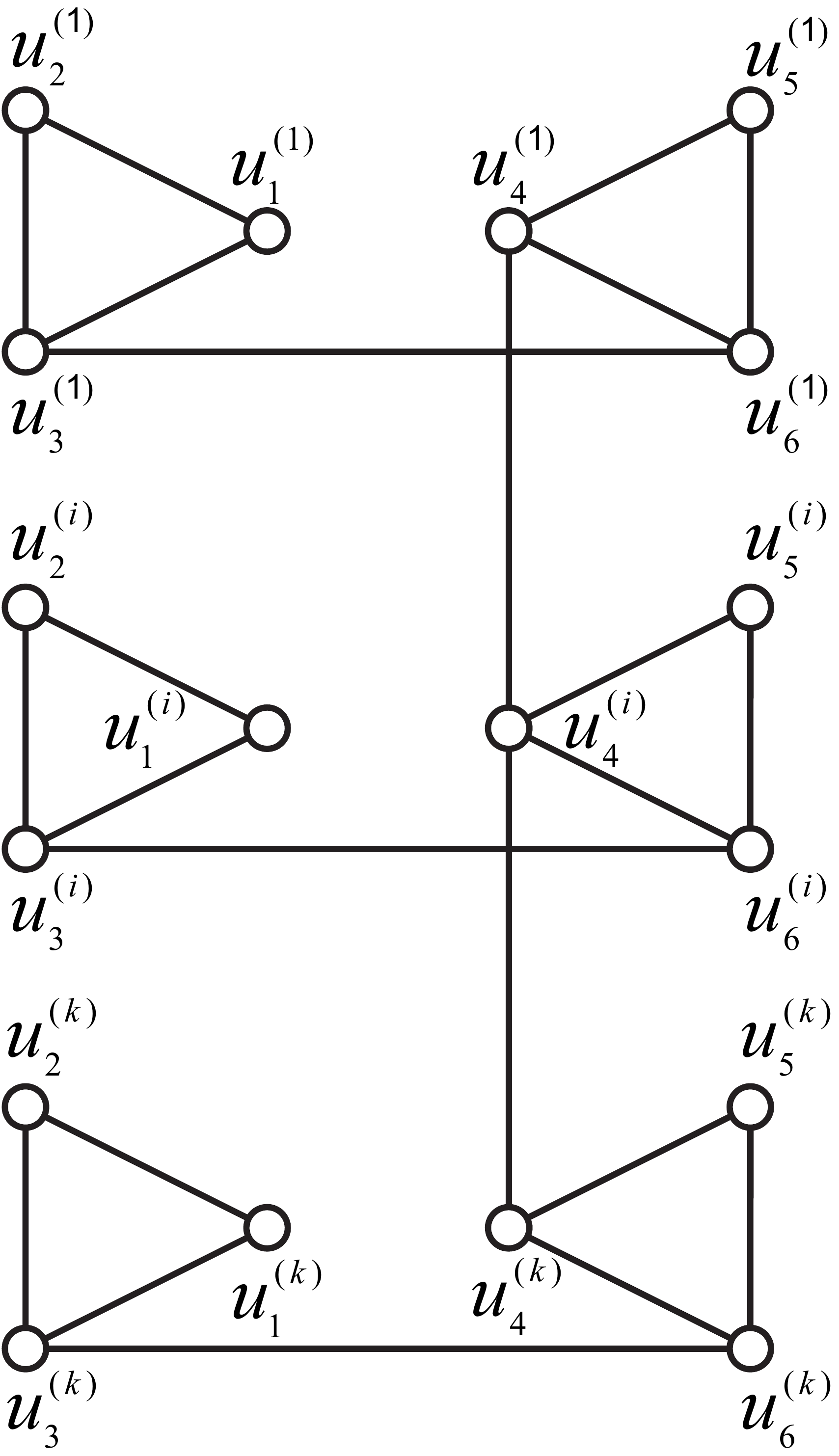}}
    \caption{Illustration for the proof of Proposition \ref{Af-gf=k}}
    \end{figure}
    \begin{Pro}
        \label{Af-gf=k}
         For any integer $k\geq 1$, $G_k$ is a matching covered graph, $gf(G_k)=3k-1$ and $Af(G_k)=4k-1$. Thus, $gf(G_k)-Af(G_k)=-k$.
    \end{Pro}

    \begin{proof}
    Take a perfect matching $M$ of $G_k$ as       $M=\left\{u_1^{(i)}u_4^{(i)},u_2^{(i)}u_5^{(i)},u_3^{(i)}u_6^{(i)}|1\leq i\leq k\right\}$.
    We can see that each edge of $G_k$ belongs to an $M$-alternating cycle in $G_k$, so $G_k$ is matching covered.
    From Theorem \ref{N.P.M}, $M$ is a nice perfect matching of $G_k$, so
    $Af(G_k)=\frac{2e(G_k)-v(G_k)}{4}=\frac{2(11k-2)-6k}{4}=4k-1$.

     Now we consider the global forcing number of $G_k$. Set
    \begin{displaymath}
        F=\left\{u_1^{(i)}u_4^{(i)},u_2^{(i)}u_5^{(i)}|1\leq i\leq k\right\}\cup \left\{u_1^{(i)}u_1^{(i+1)}|1\leq i\leq k-1\right\}.
    \end{displaymath}
    Because all cycles of $G_k-F$ are triangles (see Fig. \ref{f7_b}), it contains no nice cycles of $G_k$, so $gf(G_k))\leq \left|F\right|=3k-1$.

    Next, we prove that $gf(G_k)\geq 3k-1$ by induction on $k$. If $k=1$, then $G_k$ is a triangular prism. We know $gf(G_k)\geq 2$. Otherwise, $G_k$ has a global forcing set of a single edge, which belongs to at most two of three nice cycles  $u_1^{(1)}u_2^{(1)}u_5^{(1)}u_4^{(1)}u_1^{(1)}$, $u_2^{(1)}u_3^{(1)}u_6^{(1)}u_5^{(1)}u_2^{(1)}$ and $u_1^{(1)}u_3^{(1)}u_6^{(1)}u_4^{(1)}u_1^{(1)}$, a contradiction. We now consider $G_k$ with $k\geq 2$. Suppose that for such graphs with less than $k$ triangular prisms, the assertion holds. Let
    \begin{align}
        F_1&=E(T_1)\cup \left\{u_1^{(1)}u_1^{(2)},u_4^{(1)}u_4^{(2)}\right\},\nonumber\\
        F_k&=E(T_k)\cup \left\{u_1^{(k-1)}u_1^{(k)},u_4^{(k-1)}u_4^{(k)}\right\},\nonumber\\
        F_i&=E(T_i)\cup \left\{u_1^{(i-1)}u_1^{(i)},u_1^{(i)}u_1^{(i+1)},u_4^{(i-1)}u_4^{(i)},u_4^{(i)}u_4^{(i+1)}\right\},\ {\rm where} \ 2\leq i\leq k-1.\nonumber
    \end{align}
    Then $E(G_k)=\bigcup\limits_{i=1}^{k}F_i$. Let $S$ be a minimum global forcing set of $G_k$. We claim $\left|S\cap F_1\right|\geq 3$ or $\left|S\cap F_k\right|\geq 3$ or $\left|S\cap F_i\right|\geq 4$ for some $2
    \leq i\leq k-1$. Otherwise, $\left|S\cap F_1\right|\leq 2$, $\left|S\cap F_k\right|\leq 2$
   and for all $2\leq i\leq k-1$, $\left|S\cap F_i\right|\leq 3$. In this case, we can construct a nice cycle of $G_k$ not passing through any edge of $S$.
   Let $G'=G_k-T_k$. By induction hypothesis, $gf(G')\geq 3k-4$.
    Since $G'$ is a nice subgraph of $G_k$, by Lemma \ref{nice subgraph}, $S|_{G'}$ is a global forcing set of $G'$. We have
   \begin{eqnarray}
      3k-4&\leq& gf(G')\leq\left|S\cap E(G')\right|\nonumber\\
           &\leq& |\bigcup\limits_{i=1}^{k-1}(S\cap F_i)|
              \leq \sum\limits_{i=1}^{k-1}|S\cap F_i|
              \leq 2+3(k-2)=3k-4\nonumber.
   \end{eqnarray}
   So $\left|S\cap F_1\right|=2$, $\left|S\cap F_i\right|=3$ and $(S\cap F_{i-1})\cap(S\cap F_i)=\varnothing$ for all $2\leq i\leq k-1$. For $G_k-T_1$, in an analogous argument, we get $\left|S\cap F_k\right|=2$ and $(S\cap F_{k-1})\cap(S\cap F_k)=\varnothing$. Thus $S\subseteq \bigcup\limits_{i=1}^{k}E(T_i)$.
   For $i=1$ and $k$, since $\left|S\cap E(T_i)\right|=2$, we can take one path $P_i$ in $T_i$ from three edge-disjoint paths between $u_1^{(i)}$ and $u_4^{(i)}$ such that $S\cap E(P_i)=\varnothing$. We can check that such two paths $P_1$ and $P_k$ with two paths $P=u_1^{(1)}u_1^{(2)}\ldots u_1^{(k)}$ and $P'=u_4^{(1)}u_4^{(2)}\ldots u_4^{(k)}$ form a nice cycle of $G_k$, a contradiction. So the claim is verified.

    If $\left|S\cap F_k\right|\geq 3$ (similarly for the case $|S\cap F_1|\geq 3$), as before by the induction hypothesis we have
   \begin{displaymath}
      \left|S\right|=\left|S\cap F_k\right|+\left|S\cap E(G')\right|\geq 3+(3k-4)=3k-1.
   \end{displaymath}
    So suppose that there is an integer $2\leq i\leq k-1$ such that $\left|S\cap F_i\right|\geq 4$.
   Then $G_k-T_i$ has two components, denoted by $G_1'$ and $G_2'$. Using the induction hypothesis to $G'_1$ and $G'_2$, $gf(G'_1)+gf(G'_2)\geq 3(k-1)-2$. Further,
   \begin{displaymath}
      \left|S\right|=\left|S\cap F_i\right|+\left|S\cap E(G'_1)\right|+\left|S\cap E(G'_2)\right|\geq 4+gf(G'_1)+gf(G'_2)\geq 4+3(k-1)-2=3k-1.
   \end{displaymath}
   In conclusion, we get $gf(G_k)=3k-1$. Thus $gf(G_k)-Af(G_k)=-k$.
\end{proof}

\end{document}